\documentclass[11pt]{amsart}
\usepackage{amssymb,amsmath,amsfonts,amscd,euscript}
\usepackage{hyperref}
\newcommand{\nc}{\newcommand}

\numberwithin{equation}{section}
\newtheorem{thm}{Theorem}[section]
\newtheorem{prop}[thm]{Proposition}
\newtheorem{lem}[thm]{Lemma}
\newtheorem{cor}[thm]{Corollary}
\theoremstyle{remark}
\newtheorem{rem}[thm]{Remark}
\newtheorem{definition}[thm]{Definition}

\nc{\gl}{\mathfrak{gl}}
\nc{\mosp}{\mathfrak{osp}}
\nc{\msl}{\mathfrak{sl}}
\nc{\mslth}{\widehat{\msl}_2}
\nc{\GL}{\mathfrak{GL}}
\nc{\g}{\mathfrak{g}}
\nc{\gh}{\widehat\g}
\nc{\h}{\mathfrak{h}}
\nc{\la}{\lambda}
\nc{\al}{\alpha }
\nc{\be}{\beta }
\nc{\ve}{\varepsilon }
\nc{\om}{\omega }

\nc{\ta}{\theta}
\nc{\ch}{{\mathop {\rm ch}}}
\nc{\Tr}{{\mathop {\rm Tr}\,}}
\nc{\Id}{{\mathop {\rm Id}}}
\nc{\ad}{{\mathop {\rm ad}}}
\nc{\bra}{\langle}
\nc{\ket}{\rangle}
\nc{\x}{{\bf x}}
\nc{\bm}{{\bf m}}
\nc{\bs}{{\bf s}}
\nc{\ba}{{\bf a}}
\nc{\bb}{{\bf b}}
\nc{\bk}{{\bf k}}
\nc{\bp}{{\bf p}}
\nc{\pa}{\partial}
\nc{\ld}{\ldots}
\nc{\cd}{\cdots}
\nc{\hk}{\hookrightarrow}
\nc{\T}{\otimes}
\nc{\gr}{\mathrm{gr}}
\nc{\ov}{\overline}
\nc{\cO}{\mathcal O}
\nc{\mgl}{\mathfrak{gl}}
\nc{\U}{\mathrm U}
\nc{\V}{\EuScript V}
\nc{\cL}{\mathcal{L}}
\nc{\Res}{\mathrm{Res\ }}

\newcommand{\bC}{{\mathbb C}}

\newcommand{\bZ}{{\mathbb Z}}

\newcommand{\fh}{{\mathfrak h}}

\newcommand{\fg}{{\mathfrak g}}

\newcommand{\fn}{{\mathfrak n}}

\newcommand{\qb}[2]{\binom{#1}{#2}_{\! q}}
\newcommand{\qsb}[2]{\binom{#1}{#2}_{\! q^2}}

\begin{document}
\title[]
{Weyl modules for $\mosp(1,2)$ and nonsymmetric Macdonald polynomials}

\author{Evgeny Feigin}
\address{Evgeny Feigin:\newline
Department of Mathematics,\newline
National Research University Higher School of Economics,\newline
Vavilova str. 7, 117312, Moscow, Russia,\newline
{\it and }\newline
Tamm Theory Division, Lebedev Physics Institute
}
\email{evgfeig@gmail.com}
\author{Ievgen Makedonskyi}
\address{Ievgen Makedonskyi:\newline
Department of Mathematics, \newline
National Research University Higher School of Economics, AG Laboratory, \newline
Vavilova str. 7, 117312, Moscow, Russia
{\it and }\newline
Departments of Mechanics and Mathematics
Kiev Shevchenko National University
Vladimirskaya, 64, Kiev, Ukraine.
}
\email{}

\begin{abstract}
The main goal of our paper is to establish a connection between the Weyl modules of the current Lie 
superalgebras (twisted and untwisted) attached to $\mosp(1,2)$ and the nonsymmetric Macdonald polynomials of types $A_2^{(2)}$
and $A_2^{(2)\dagger}$. We compute the dimensions and construct bases of the Weyl modules. 
We also derive explicit formulas for the $t=0$ and $t=\infty$ specializations
of the nonsymmetric Macdonald polynomials. We show that the specializations can be described in terms
of the Lie superalgebras action on the Weyl modules.
\end{abstract}

\maketitle

\section*{Introduction}
The Weyl modules play important role in the representation theory of infinite-dimensional
Lie algebras (see \cite{CP,CL,FoLi1,FeLo1}). These representations pop up in various problems of representation theory 
and its applications \cite{CFK,FeLo2,FoLi1,FoLi2,SVV}.
The most important for us is the established in many cases isomorphism between the Weyl modules and Demazure modules
(at least, in types $A,D,E$)
for integrable representations of the affine Kac-Moody algebras. The Demazure modules are known to provide
finite-dimensional approximation of the infinite-dimensional modules of the affine Kac-Moody Lie algebras and hence so do the 
Weyl modules.
Yet another consequence from the link between the Demazure and Weyl modules is that the characters of both (in the
level one case) can be expressed in terms of the nonsymmetric Macdonald polynomials (see \cite{I,S}).
The nonsymmetric Macdonald polynomials (see \cite{M1,M2}) are rational functions in parameters $q$ and $t$ and 
the characters of the Demazure mdoules are equal to the $t=0$ specialization.  
It was conjectured recently (see \cite{CO1,CO2,FM}) that the $t=\infty$ specialization also has representation theoretic
realization in terms of the PBW filtration (see \cite{FFL1,FFL2} for the case of simple Lie algebras).   
 
Now let us turn to the case of superalgebras. It is not clear what should be an appropriate definition of a Demazure module
for superalgebras. However, the Weyl modules are perfectly well defined (see e.g. \cite{CLS}). So there are two natural questions
here. The first one is to compute the characters of the Weyl modules for affine superalgebras and to find a connection
with some super analogues of the nonsymmetric Macdonald polynomials. The second question is to figure out if a limit
of the Weyl modules (when the highest weight grows) does exist. In this paper we consider the smallest Lie superalgebra
$\mosp(1,2)$, which plays in the super theory a role similar to that of the Lie algebra $\msl_2$ in the theory
of classical simple Lie algebras. The Weyl modules in this case are parametrized by a non-negative integer $n$;
we denote the corresponding $\mosp(1,2)\T\bC[t]$ module by $W_{-n}$. We prove the following theorem.
\begin{thm}
$W_{-n}$ as $\mosp(1,2)$-module is isomorphic to the tensor product of $n$ copies of $3$-dimensional irreducible $\mosp(1,2)$-module.
Moreover, the  $\mosp(1,2)\T\bC[t]$-module structure is given by the graded tensor product (fusion product) construction.
\end{thm}

We show that $W_{-n}$ can be filtered by the Weyl modules for $\msl_2\T \bC[t]$. This allows us to construct bases and compute the characters of
$W_{-n}$. Our next goal is to relate the characters of the Weyl modules to the nonsymmetric Macdonald polynomials
(\cite{Ch1,Ch2}).
Using the Ram-Yip formula for the nonsymmetric Macdonald polynomials (see \cite{RY,OS}),  we prove the following theorem.
\begin{thm}
Let $E_{-n}^{A_2^{(2)\dagger}}(x,q,t)$ be the nonsymmetric Macdonald polynomials of type $A_2^{(2)\dagger}$. Then the character of 
$W_{-n}$ is given by $E_{-n}^{{A_2^{(2)\dagger}}}(x,q,0)$ and the specialization $E_{-n}^{{A_2^{(2)\dagger}}}(x,q,\infty)$ coincides
with the PBW twisted character of $W_{-n}$. 
\end{thm}      

We close the introduction with several remarks. 

First, the $\mosp(1,2)$ current algebra  has the 
twisted analogue $\mosp(1,2)[t]^\sigma$. We study all the questions described above in the twisted case as well. 
In particular, we establish a connection
with the specializations of the nonsymmetric Macdonald polynomials of type $A_2^{(2)}$.
We note that both algebras  $\mosp(1,2)[t]$ and $\mosp(1,2)[t]^\sigma$ are Borel's subalgebras in the affine 
superalgebra $\widehat{\mosp(1,2)}$.

Second, in both twisted and untwisted cases we define the positive $n$ versions of the Weyl modules $W_n$.
We also make a link to the Macdonald polynomials. 

Third, we show that there exist embeddings of $\mosp(1,2)\T\bC[t]$-modules $W_{-n}\subset W_{-n-1}$ and we compute the
character of the (infinite-dimensional) limit. We note that we do not know if there is a structure of the representation of 
a larger algebra on this limit. 

Finally, let us mention that in the Macdonald polynomials part of our paper we follow the ideas and methods of the paper \cite{OS}.
In Appendix \ref{OrrSh} we describe the most important for us ingredients of the approach of Orr and Shimozono. 

The paper is organized as follows:\\
In Section \ref{Weyl} we study the Weyl modules for $\mosp(1,2)\T\bC[t]$ and their twisted version.\\
In Section \ref{Macdonald} we derive explicit formulas for the types $A_2^{(2)}$ and ${A_2^{(2)\dagger}}$ Macdonald
polynomials.\\
In Section \ref{WM} we establish a connection between the two parts of the story.

\section{Weyl modules}\label{Weyl}
\subsection{The classical case}
Let $D_{-n}$, $n\ge 0$ be the Weyl modules for the current algebra $\msl_2[t]=\msl_2\T\bC[t]$. They are defined as 
finite-dimensional cyclic 
modules with cyclic vector $d_{-n}$, subject to the conditions
\[
hd_{-n}=-nd_{-n},\ f\T\bC[t].d_{-n}=0,\ h\T t\bC[t].d_{-n}=0,
\]
where $e,h,f$ for the standard basis of $\msl_2$. 
These modules are known to be $2^n$-dimensional with a monomial basis
\[
e_{a_1}\dots e_{a_k}d_{-n},\ 0\le a_1\le \dots\le a_k\le n-k.
\] 
For a graded vector space $M=\bigoplus_{s\ge 0} M_s$ with an action of the operator $h$
We define the character ${\rm ch} M(x,q)$ as $\sum_{s\ge 0} q^s{\rm tr}(x^h|_{M_s})$. 
The character of $D_{-n}$ is equal to $\sum_{k=0}^n x^{-n+2k}\qb{n}{k}$, where
the $q$-binomial coefficients are given by the formula 
\[
\qb{n}{m}=\frac{(1-q)(1-q^2)\dots(1-q^n)}{(1-q)\dots(1-q^m)(1-q)\dots(1-q^{n-m})}.
\]

The modules $D_{-n}$ are known to be isomorphic to the graded tensor product (the fusion product \cite{FeLo1}) of $n$
copies of standard $2$-dimensional $\msl_2$-modules. Moreover, $D_{-n}$ is isomorphic to a Demazure module in the
basic level one representation of the affine Lie algebra $\mslth$. In particular, there exist embeddings of $\msl_2\T\bC[t]$-modules 
\[D_{0}\subset D_{-2}\subset D_{-4}\subset \dots,\qquad D_{-1}\subset D_{-3}\subset D_{-5}\subset \dots\]
and the inductive limits are isomorphic to the integrable $\mslth$ modules of level $1$. 
We have the explicit formula for the characters of the limits:
\[
\ch \lim_{n\to\infty} D_{-2n}=\sum_{k\in\bZ} x^{2k}\frac{q^{k^2}}{(q)_\infty},\
\ch \lim_{n\to\infty} D_{-2n-1}=\sum_{k\in\bZ} x^{2k+1}\frac{q^{k(k+1)}}{(q)_\infty}.
\]

\subsection{Weyl modules for superalgebras}
Our references here are \cite{P,Mus1,Mus2}.
The Lie superalgebra $\mosp(1,2)$ is isomorphic to the direct sum $\msl_2\oplus\pi_1$, where
$\msl_2$ is the even part and the two-dimensional $\msl_2$ module $\pi_1$ if the odd part.
Let $e,h,f$ be the standard basis of $\msl_2$ and let $g^+,g^-$ be the basis of $\pi_1$. One has
the nontrivial commutation relations
\begin{gather*}
[e,f]=h,\ [h,e]=2e,\ [h,f]=-2f,\\
[h,g^+]=g^+,\ [h,g^-]=-g^-,\ [g^+,g^-]_+=h, \\ [g^+,g^+]_+=2e,\ [g^-,g^-]_+=-2f,\ [f,g^+]=g^-,\ [e,g^-]=-g^+.
\end{gather*}
We have the Cartan decomposition
\[
\mosp(1,2)=\fn^-\oplus\fh\oplus\fn^+,\ \fn^-={\rm span}(f,g^-),\ \fn^+={\rm span}(e,g^+),\ \fh={\mathrm span}(h).
\]

We consider the current algebra $\mosp(1,2][t]=\mosp(1,2)\T \bC[t]$ and its Weyl module $W_{-n}$ defined as the cyclic 
module with a generator $w_{-n}$ subject to the relations
\begin{gather*}
(\fn^-\oplus\fh)\T t\bC[t]. w_{-n}=0,\ \ (\fn^-\T 1).w_{-n}=0,\\ h_0.w_{-n}=-nw_{-n},\ \ (e\T 1)^{n+1}.w_{-n}=0.
\end{gather*}
For $x\in\mosp(1,2)$ we denote by $x_i\in\mosp(1,2)[t]$ the element $x\T t^i$. 

\begin{lem}\label{untwbas}
One has $\mosp(1,2)\T t^n\bC[t].w_{-n}=0$ and $W_{-n}$ is spanned by the monomials of the form
\begin{gather}\label{basis}
e_{a_1}\dots e_{a_s}g^+_{b_1}\dots g^+_{b_k}w_{-n},\\ \nonumber 
0\le b_1<\dots<b_k\le n-1,\ 0\le a_1\le\dots\le a_s\le n-k-s. 
\end{gather}
\end{lem}
\begin{proof}
The condition $\msl_2\T t^n\bC[t] w_{-n}=0$ follows from the results on the Weyl modules for $\msl_2$
(see e.g. \cite{CL}).
Now if $e_iw_{-n}=0$, then $g^-_0e_iw_{-n}=g^+_iw_{-n}=0$ and similarly for $g^-_iw_{-n}$.

We note that since $[g^+_i,g^+_j]_+=2e_{i+j}$, we have
\[
W_{-n}=\sum_{0\le b_1<\dots<b_k\le n-1} U(\msl_2\T\bC[t])g^+_{b_1}\dots g^+_{b_k}w_{-n}.
\]
We introduce a partial order on the monomials $g^+_{b_1}\dots g^+_{b_k}$, $0\le b_1<\dots<b_k\le n-1$, saying that
for two different monomials
$g^+_{b_1}\dots g^+_{b_k}<g^+_{c_1}\dots g^+_{c_l}$ if $k>l$ or ($k=l$ and $b_i\ge c_i$, $i=1,\dots,k$).
Let us totally order the monomials  $g^+_{b_1}\dots g^+_{b_k}$, $0\le b_1<\dots<b_k\le n-1$ as $m_1,m_2,\dots,m_N$ in such a way
that if $i<j$ then $m_i<m_j$ (in the sense of the partial order above). 

Now let us introduce an increasing filtration $F_i$ on $W_{-n}$ by
\[
F_i=\sum_{j=1}^i U(\msl_2\T\bC[t])m_iw_{-n}.
\]
We claim that the monomials \eqref{basis} span the associated graded space ${\rm gr} F_\bullet$.
Indeed, let $\overline{m_iw_{-n}}$ by the image of $m_iw_{-n}$ in the associated graded space. Then
for $m_i=g^+_{b_1}\dots g^+_{b_k}$ we have 
\[
(\fn^-\oplus\fh)\T t\bC[t].\overline{m_iw_{-n}}=0,\ (\fn^-\T 1). \overline{m_iw_{-n}}=0,\ (h\T 1).\overline{m_iw_{-n}}=-(n+k)w_{-n}.
\]
Hence all the relations for the $\msl_2\T\bC[t]$ Weyl module with lowest weight $-n+k$ are satisfied, which implies the claim of the lemma.
\end{proof}

\begin{lem}\label{-0}
The number of elements in the set \eqref{basis} is equal to $3^n$ and its character is given by
\[
\sum_{k=0}^n q^{k(k-1)/2}\qb{n}{k}\sum_{s=0}^{n-k} x^{-n+k+2s} \qb{n-k}{s}.
\] 
\end{lem}
\begin{proof}
Straightforward.
\end{proof}

Now let us give the twisted version of the lemma above. We replace the current algebra $\mosp(1,2)[t]$ with
its twisted analogue 
\[
\mosp(1,2)[t]^{\sigma}=\bigoplus_{i=0}^\infty \msl_2\T t^{2i}\oplus \bigoplus_{i=0}^\infty \pi_1\T t^{2i+1}. 
\]
This is again a Lie superalgebra. We define its Weyl module $W_{-n}^\sigma$ 
as the cyclic 
integrable with respect to $\msl_2$ module with a generator $w^\sigma_{-n}$ subject to the relations
\[
f_{2i}. w^\sigma_{-n}=0,\ g^-_{2i+1}. w^\sigma_{-n}=0,\ h_{2i+2}. w^\sigma_{-n}=0, \text{ for } i\ge 0
\]
and $h_0.w^\sigma_{-n}=-nw^\sigma_{-n}$, $e_0^{n+1}.w_{-n}=0$.

\begin{lem}
One has $\msl_2\T t^{2n}\bC[t^2].w_{-n}=0$ and $\pi_1\T t^{2n+1}\bC[t^2].w_{-n}=0$. 
$W^\sigma_{-n}$ is spanned by the monomials of the form
\begin{gather}\label{twbasis}
e_{a_1}\dots e_{a_s}g^+_{b_1}\dots g^+_{b_k}w_{-n},\\ \nonumber 1\le b_1<\dots<b_k\le 2n-1,\ 0\le a_1\le\dots\le a_s\le 2(n-k-s), 
\end{gather}
where $a_i$ are even and $b_j$ are odd.
\end{lem}

\begin{lem}\label{tw-0}
The number of elements in the set \eqref{twbasis} is equal to $3^n$ and its character is given by
\[
\sum_{k=0}^n q^{k^2}\qsb{n}{k}\sum_{s=0}^{n-k} x^{-n+k+2s} \qsb{n-k}{s}.
\] 
\end{lem}

\subsection{The graded tensor products for superalgebras}
We start with the untwisetd case.

Let $\fg=\fg_{\bar 0}\oplus\fg_{\bar 1}$ be a Lie superalgebra with $\fg_{\bar 0}$ being the even part and
$\fg_{\bar 1}$ being the odd part. For a $\fg$-module $X$ we denote by $X_{\bar 0}$ its even part and by $X_{\bar 1}$
its odd part. Let $V$ and $W$ be  cyclic $\fg$ modules 
with cyclic vectors $v$ and $w$; in what follows we always assume that the cyclic vectors are even.
The tensor product of $V$ and $W$ is defined by the formula
\[
g (x\T y) = gx\T y + (-1)^{ab} x\T gy,\ g\in\fg_{\bar a}, x\in V_{\bar b}.   
\]  
 
Let $z_1,\dots,z_n$ be a collection of pairwise distinct complex numbers and let $V^1,\dots, V^n$
be cyclic representations of $\fg$ with cyclic vectors $v^1,\dots,v^n$. Let $V^i(z_i)$ be the 
evaluation representations of $\fg\T\bC[t]$, where $x\T t^k$ acts as $z_i^kx$. 
\begin{lem}
The tensor product $\bigotimes_{i=1}^n V^i(z_i)$ is cyclic $\fg[t]$-module with cyclic vector $\otimes_{i=1}^n v^i$.
\end{lem}
\begin{proof}
Let $x\in\fg_{\bar 0}$. Then the operator $x\T t^k$ acts on the tensor product $\bigotimes V^i(z_i)$ by
the usual tensor product formula for representations of Lie algebras. Therefore all the operators 
$$x(i)=\underbrace{{\rm Id}\otimes \dots {\rm Id}}_{i-1}\otimes x\otimes {\rm Id}\otimes\dots {\rm Id}$$
on $\bigotimes_{i=1}^n V^i(z_i)$
can be written as linear combinations of the operators $x\T t^k$ (via the Vandermond determinant).

Now assume that $x\in\fg_{\bar 1}$. Then one has
\begin{multline*}
(x\T t^k)(u_1\otimes\dots\otimes u_n)=\\
(z_1^kx(1)+z_2^k (-1)^{\deg u_1}x(2)+ \dots +z_n^k(-1)^{\deg u_1+\dots+\deg u_{n-1}}z_n^kx(n))(u_1\otimes\dots\otimes u_n).
\end{multline*}
Therefore for any $i=1,\dots,n$ the operator $(-1)^{\deg u_1+\dots+\deg u_{i-1}}x(i)$ can be expressed as linear combination
of the operators $x\T t^j$, $0\le j\le n-1$ (the case $i=1$ corresponds to just $x(1)$). Therefore 
the operators $x\T t^j$, $0\le j\le n-1$ do generate the whole tensor product $\bigotimes_{i=1}^n V^i(z_i)$
acting on the tensor product of cyclic vectors.
\end{proof}

The universal enveloping algebra $\U(\fg[t])$ has natural grading coming for the degree of $t$,
$\U(\fg[t])=\bigoplus_{s\ge 0} \U(\fg[t])_s$ (for example, $\U(\fg[t])_0=\U(\fg))$.
Let us introduce the increasing filtration $F_s$ on the tensor product $\bigotimes_{i=1}^n V^i(z_i)$ as follows:
\[
F_s=\U(\fg[t])_s(v^1\otimes\dots\otimes v^n).
\] 
The associated graded space is cyclic $\U(\fg[t])$ module. An important feature is that it is now equipped with
the additional grading. We denote the graded module by $V^1(z_1)*\dots *V^n(z_n)$.

Let $V$ be the irreducible 3-dimensional representation of $\mosp(1,2)$. 
\begin{thm}
For any pairwise distinct $z_1,\dots,z_n$ the graded tensor product $V(z_1)*\dots *V(z_n)$ is isomorphic
to the Weyl module $W_{-n}$ as  $\mosp(1,2)[t]$ modules.
\end{thm}
\begin{proof}
It is easy to see that all the defining relations of the Weyl module do hold in the graded tensor product.
Therefore, we have a surjection  $W_{-n}\to V(z_1)*\dots *V(z_n)$. Since the dimension of the right hand side is
$3^n$ and this is the upper estimate for the dimension of the left hand side (see Lemma \ref{-0}), the surjection is the isomorphism. 
\end{proof}

\begin{cor}
The graded tensor product $V(z_1)*\dots *V(z_n)$ does not depend (as $\mosp(1,2)[t]$-module) on the (pairwise distinct) parameters $z_i$.  
\end{cor}

\begin{cor}
Vectors \eqref{basis} form a basis of the Weyl module $W_{-n}$.
\end{cor}

Now let us consider the twisted algebra $\mosp(1,2)[t]^\sigma$. 
For a complex number $z$ one defines the 3-dimensional evaluation $\mosp(1,2)[t]^\sigma$-module $V^\sigma(z)$ via the same formula as above.
We have the following theorem.
\begin{thm}
Assume that the numbers $z_1,\dots,z_n$ satisfy the conditions $z^2_i\ne z^2_j$ for $i\ne j$. Then
the graded tensor product $V^\sigma(z_1)*\dots *V\sigma(z_n)$ is well defined and is isomorphic
to the Weyl module $W^\sigma_{-n}$ as  $\mosp(1,2)[t]$ modules.
\end{thm}
\begin{proof}
The only difference with the untwisted case is the condition $z_i^2\ne z_j^2$, which guaranties 
the cyclicity of the tensor product of evaluation modules in the twisted case.
\end{proof}
\begin{cor}
The graded tensor product $V^\sigma(z_1)*\dots *V^\sigma(z_n)$ does not depend (as $\mosp(1,2)[t]^\sigma$-module) 
on the parameters $z_i$, satisfying $z_i^2\ne z_j^2$ for all $i\ne j$.  
Vectors \eqref{twbasis} form a basis of the Weyl module $W^\sigma_{-n}$.
\end{cor}

\subsection{The positive $n$ case}
In this subsection we define the modules $W_n$ for $n>0$. 
We first consider the untwisted case.

We define vector $w_n=e_0^nw_{-n}\subset W_{-n}$. We note that there is a symmetry (the $A_1$ Weyl group action) on $W_{-n}$
interchanging $e$ with $f$ and $g^+$ with $g^-$. This symmetry sends $w_{-n}$ to $w_n$ and vice versa.
Therefore, the module $W_{-n}$ enjoys the basis of the form
\begin{equation}
f_{a_1}\dots f_{a_s}g^-_{b_1}\dots g^-_{b_k}w_n,\ 0\le b_1<\dots<b_k\le n-1,\ 0\le a_1\le\dots\le a_s\le n-k-s. 
\end{equation}
 
Let $W_n=\U(\fn^-\T t\bC[t]).w_n\subset W_{-n}$, so $W_n$ is a module for the shifted Borel subalgebra, generated
by $g^-_1$ and $g^+_0$.    
In other words, the module $W_n$ is generated from the vector $w_n$ by the action
of the operators $f_i$ and $g^-_i$, $i>0$. 
\begin{prop}
$\dim W_n=3^{n-1}$. The vectors 
\begin{gather}\label{shift}
f_{a_1}\dots f_{a_s}g^-_{b_1}\dots g^-_{b_k}w_{-n},\\ \nonumber 
1\le b_1<\dots<b_k\le n-1,\ 1\le a_1\le\dots\le a_s\le n-s-k 
\end{gather}
form a basis of $W_n$. 
\end{prop}
\begin{proof}
The vectors \eqref{shift} belong to the basis of the module $W_{-n}$ and hence are linear independent.
Now we know that for the $2^{n-k}$ dimensional Weyl module for $\msl_2$ the part generated by $f_1,f_2,\dots$ from the lowest
weight vector has basis of the form $f_{a_1}\dots f_{a_s}$, where $1\le a_1\le\dots\le a_s\le n-s-k$.
Now using filtration from the proof of Lemma \ref{untwbas} we obtain that \eqref{shift} is indeed a basis.  
\end{proof}

\begin{cor}\label{+0}
The character of $W_n$ is equal to
\[
\sum_{k=0}^{n-1} q^{k(k+1)/2}\qb{n-1}{k}\sum_{s=0}^{n-k-1} x^{n-k-2s}q^s \qb{n-k-1}{s}.
\]
\end{cor}

Now let us work out the twisted case.

We define vector $w^\sigma_n=e_0^nw^\sigma_{-n}\subset W^\sigma_{-n}$. Let
\[
W^\sigma_n=\U((\mosp(1,2)\T t\bC[t])^\sigma).w^\sigma_n\subset W_{-n},
\] 
i.e. $W^\sigma_n$ is a cyclic module for the shifted Borel subalgebra, generated by $g^-_1$ and $e_0$. 
The module $W^\sigma_n$ is generated from the vector $w_n$ by the action
of the operators $f_i$, $i=2,4,\dots,2n-2$ and $g^-_i$, $i=1,3,\dots,2n-1$. 
\begin{prop}
$\dim W_n=2\cdot 3^{n-1}$. The vectors 
\begin{gather}\label{twshift1}
f_{a_1}\dots f_{ a_s}g^-_{b_1}\dots g^-_{b_k}w^\sigma_n,\\ 
\nonumber
1\le b_1<\dots<b_k\le 2n-3,\ 2\le a_1\le\dots\le a_s\le 2(n-s-k)
\end{gather}
and
\begin{gather} 
\label{twshift2}
f_{a_1}\dots f_{ a_s}g^-_{b_1}\dots g^-_{b_{k-1}}g^-_{2n-1}w^\sigma_n,\\
\nonumber 1\le b_1<\dots<b_{k-1}\le 2n-3,\ 0\le a_1\le\dots\le a_s\le 2(n-s-k). 
\end{gather}
form a basis of $W^\sigma_n$.
\end{prop}
\begin{proof}
We first prove that the elements \eqref{twshift2} belong to $W^\sigma_n$ (this is obvious for \eqref{twshift1}, but not
for \eqref{twshift2}).
The only problem is the operator $f_0$ popping up in \eqref{twshift2}. However, it always comes multiplied by $g^-_{2n-1}$.
Therefore, we only need to check that $f_0^mg_{2n-1}^-w_n\in W_n^\sigma$ for all $m>0$. First, we note that
by the weight reason $f_0^ng_{2n-1}^-w_n=0$. Second, it suffices to prove the claim for $m=n-1$, since $W^\sigma_n$ is 
$e_0$-invariant. Third, the vectors $f_0^{n-1}g_{2n-1}^-w_n$ is proportional to $f_2^{n-1}g_1^-w_n$. Indeed, 
since the weight space containing both vectors is one-dimensional, we only need to check that $f_2^{n-1}g_1^-w_n\ne 0$.
Assume that this vector vanishes. Then $e_0^{n-1} f_2^{n-1}g_1^-w_n=0$. However, up to a nonzero constant, this vector
is equal to  $g^-_{2n-1}w_n$, which does not vanish. 

So we know that all the vectors \eqref{twshift1} and \eqref{twshift2} belong to $W_n^\sigma$. We also know that
they are linearly independent, since they belong to the basis \eqref{twbasis} of the whole space $W^\sigma_{-n}$.
Since the sets \eqref{twshift1} and \eqref{twshift2} contain $2\cdot 3^{n-1}$ elements, we are left to show that
the dimension of $W^\sigma_{-n}/W^\sigma_n\ge 3^{n-1}$. We know that relations in $W^\sigma_{-n}$ are generated by $e_0^{n+1}w_{-n}$.
Since the vector $e_0^nw_{-n}$ is trivial in the quotient, we conclude that 
\[
\dim W^\sigma_{-n}/W^\sigma_n\ge \dim W^\sigma_{-n+1}=3^{n-1}.
\]   
\end{proof}

\begin{cor}\label{tw+0}
The character of $W^\sigma_n$ is equal to
\begin{multline*}
\sum_{k=0}^{n-1} q^{k^2}\qsb{n-1}{k}\sum_{s=0}^{n-k-1} x^{n-k-2s}q^{2s} \qsb{n-k-1}{s} +\\ 
x^{-1}q^{2n-1}\sum_{k=0}^{n-1} q^{k^2}\qsb{n-1}{k}\sum_{s=0}^{n-k-1} x^{n-k-2s} \qsb{n-k-1}{s}.
\end{multline*}
\end{cor}

\subsection{The limit procedure}
In this subsection we consider the $n\to\infty$ limits of the modules $W_{-n}$ and $W_{-n}^\sigma$.
\begin{lem}\label{emb}
For $n>0$ there exists an embedding of $\mosp(1,2)[t]$ modules $W_{-n}\to W_{-n-1}$, defined by  
$w_{-n}\mapsto g^+_nw_{-n-1}$.  
\end{lem}
\begin{proof}
First, we show that the vector $g^+_nw_{-n-1}$ satisfies all the annihilation conditions for the cyclic vector of
$W_{-n}$. Clearly, $h_0g^+_nw_{-n-1}=-ng^+_nw_{-n-1}$. Now, for $x\in\fn^-$ and $k\ge 0$ one has
$x_kg^+_nw_{-n-1}=0$ (because $x_kw_{-n-1}=0$) and $[x_k,g^+_n]w_{-n-1}=0$ because
\[
[x_k,g^+_n]\subset (\fn^-\oplus\fh)\T t\bC[t]\oplus \bC g^+\T t^n\bC[t].
\]
Finally, for $k>0$ $h_kg^+_nw_{-n-1}=0$, because $[h_k,g^+_n]=g^+_{n+k}$.
We conclude that there exists a surjective homomorphism of $\mosp(1,2)[t]$-modules $W_{-n}\to \U(\mosp(1,2)[t])g^+_nw_{-n-1}$.

Second, we check that $\dim\U(\mosp(1,2)[t])g^+_nw_{-n-1}\ge 3^n$. Indeed, the vectors 
\[
e_{a_1}\dots e_{a_s}g^+_{b_1}\dots g^+_{b_k}g^+_nw_{-n-1} 
\]
with $0\le b_1<\dots<b_k\le n-1$ and $0\le a_1\le\dots\le a_s\le n-k-s$
are linearly independent, since they belong to the set of basis vectors \eqref{basis} (with $n+1$ instead of $n$).
The number of these elements is $3^n$. 
\end{proof}

Let $L=\lim_{n\to\infty} W_{-n}$ be the $\mosp(1,2)[t]$-module, obtained via the embeddings from Lemma \ref{emb}.
We define the character of $L$ as follows:
\[
\ch L(x,q)= \lim_{n\to\infty} q^{n(n-1)/2}\ch W_{-n}(x,q^{-1}).
\]
Our goal is to find a formula for the character of $L$. The following lemma is well known (see e.g. \cite{A}).
\begin{lem}\label{wedge}
$\sum_{k\ge 0} \frac{q^{k(k-1)/2}}{(q)_k}x^k=\prod_{i=0}^\infty (1+q^ix)$.
\end{lem}

\begin{thm}
$\ch L(x,q)=\prod_{i=0}^\infty (1+q^ix) \prod_{i=0}^\infty (1+q^ix^{-1})$.
\end{thm}
\begin{proof}
Recall the basis \eqref{basis} of the space $W_{-n}$. 
Let $l_0(n)\in W_{-n}$ be the vector $g^+_0g^+_1\dots g^+_{n-1}w_{-n}$. We note that the embedding $W_{-n}\to W_{-n-1}$
sends $l_0(n)$ to $l_0(n+1)$. Therefore, in the limit we obtain the vector $l_0=\lim_{n\to\infty} l_0(n)\in L$.
We note that the summand in the character of $L$ corresponding to $l_0$, is just $1$(=$x^0q^0$). It is convenient
to parametrize the images of the elements  $g^+_{b_1}\dots g^+_{b_k}w_{-n}$, $0\le b_1\le\dots\le b_k\le n-1$ in the
limit space $L$ by the set
\[
g^-_{-c_1}\dots g^-_{-c_k}l_0,\ 0\le c_1<\dots< c_k
\]
(although we do not have the action of the operators with negative powers of $t$).  
Then we obtain the basis of $L$ of the form
\begin{equation}\label{limbasis}
e_{a_1}\dots e_{a_s}g^-_{-c_1}\dots g^-_{-c_k}l_0,\ 0\le c_1<\dots< c_k,\ 0\le a_1\le\dots\le a_s\le k-s.
\end{equation}
The character of the basis vectors  \eqref{limbasis} with fixed $k$ and $s$ is equal to
\begin{multline*}
x^{-k+2s}\frac{q^{k(k-1)/2}}{(q)_k}\binom{k}{s}_{q^{-1}}=\frac{x^{-k+2s}q^{k(k-1)/2-k(s-k)}}{(q)_s(q)_{k-s}}=\\
\frac{q^{s(s-1)/2}x^{s}}{(q)_s}\times\frac{q^{(k-s)(k-s-1)/2}x^{s-k}}{(q)_{k-s}}.
\end{multline*}
Therefore 
\[
\ch L(x,q)=\sum_{s\le k}\frac{q^{s(s-1)/2}x^{s}}{(q)_s}\times\frac{q^{(k-s)(k-s-1)/2}x^{s-k}}{(q)_{k-s}}=
\prod_{i=0}^\infty (1+q^ix) \prod_{i=0}^\infty (1+q^ix^{-1})
\]
(the last equality comes form Lemma \ref{wedge}).
\end{proof}

\begin{rem}
We note that the character $\ch L(x,q)$ is equal to the (normalized) theta-function divided by the 
$\eta$-function $\prod_{i\ge 1} (1-q^i)$.
\end{rem}

Now let us turn to the twisted case.
Here we state the twisted analogues of the results from the previous subsection.
\begin{lem}
For $n>0$ there exists an embedding of $\mosp(1,2)[t]^\sigma$ modules $W^\sigma_{-n}\to W^\sigma_{-n-1}$, defined by  
$w_{-n}\mapsto g^+_{2n-1}w_{-n-1}$.  
\end{lem}

Let $L^\sigma=\lim_{n\to\infty} W^\sigma_{-n}$ be the $\mosp(1,2)[t]^\sigma$-module.
We define the character of $L$ as follows:
\[
\ch L(x,q)= \lim_{n\to\infty} q^{n^2}\ch W^\sigma_{-n}(x,q^{-1}).
\]

Let $l_0(n)=g^+_1\dots g^+_{2n-1}w^\sigma_{-n}\in W^\sigma_{-n}$ and let $l_0=\lim_{n\to\infty} l_0(n)$.
\begin{thm}
$\ch L^\sigma(x,q)=\prod_{i=0}^\infty (1+q^{2i+1}x) \prod_{i=0}^\infty (1+q^{2i+1}x^{-1})$. 
\end{thm}

\section{Nonsymmetric Macdonald polynomials of types $A_2^{(2)}$ and ${A_2^{(2)\dagger}}$}\label{Macdonald}

Nonsymmetric Koornwinder polynomials of types $A_{2n}^{(2)}$ and ${A_{2n}^{(2)\dagger}}$ (see e.g. \cite{OS,H,RY})
are rational functions depending on a parameter $q$ and five independent Hecke parameters $v_s$, $v_l$, $v_2$, $v_0$, $v_z$.
Nondegenerate nonsymmetric Macdonald polynomials of types $A_2^{2n}$ (${A_2^{2n\dagger}}$ correspondingly) are defined as specializations
of Koornwinder polynomials at $v_2 \mapsto 1$ ($v_z \mapsto 1$ correspondingly) and equal Hecke parameters $v_i$'s 
(we denote them by $v$).  Functions thus obtained are rational functions in
variables $x, q, t=v^2$, more precisely they belong to 
$\mathbb{Z}\left( q, t \right)\left[ x \right]$. We study the limits $v \rightarrow 0$ and $v \rightarrow \infty$ for $n=1$.

\subsection{Ram-Yip formula for type $A_2^{(2)} ({A_2^{(2)\dagger}})$}\label{RY}
We compute specializations of nonsymmetric Macdonald polynomials using methods from the papers \cite{RY,OS}.
We use the so called alcove walks, which are certain paths on the set of alcoves. We don't give the general definition of
an alcove walk: one can find it in papers \cite{RY,OS}. However we give an explicit construction in our case.

We consider the real line $\mathbb{R}$ and the set of alcoves $(i, i+1), i \in \mathbb{Z}$ and
label the walls of alcoves by simple reflections such that the wall $i$ is labeled by $s_1$ if $i$ is even and
by $s_0$ if $i$ is odd.

The Weyl group $W=\langle s_1 \rangle\star\langle s_0 \rangle$ (the
free product of two cyclic groups of  order $2$) acts on the set of alcoves simply transitively. 
We identify $W$ with the set of alcoves:
$s_1$ acts as a reflection in the wall $0$ and $s_0$ acts as a reflection in the wall $1$. Hence $2X=s_1s_0$ acts as a shift by $2$.
We  use additive notation for the group generated by $2X$ (i. e. we write elements of this group as $2nX, n \in \mathbb{Z}$).
Any element of $W$ can be written in the form $2nX s_1^b, b \in \lbrace 0,1 \rbrace$. The following picture illustrates 
the procedure:

\begin{picture}(280.00,50.00)(-5,00)
\put(5.00,20.00){\line(1,0){275}}
\put(35.00,10.00){\line(0,1){20}}
\put(65.00,10.00){\line(0,1){20}}
\put(95.00,10.00){\line(0,1){20}}
\put(125.00,10.00){\line(0,1){20}}
\put(155.00,10.00){\line(0,1){20}}
\put(185.00,10.00){\line(0,1){20}}
\put(215.00,10.00){\line(0,1){20}}
\put(245.00,10.00){\line(0,1){20}}
\put(275.00,10.00){\line(0,1){20}}

 \put(35,5){\makebox(1,1){$-2$}}
 \put(65,5){\makebox(1,1){$-1$}}
 \put(95,5){\makebox(1,1){$0$}}
 \put(125,5){\makebox(1,1){$1$}}
 \put(155,5){\makebox(1,1){$2$}}
 \put(185,5){\makebox(1,1){$3$}}
 \put(215,5){\makebox(1,1){$4$}}
 \put(245,5){\makebox(1,1){$5$}}
 \put(275,5){\makebox(1,1){$6$}}

 \put(35,35){\makebox(1,1){$s_1$}}
 \put(65,35){\makebox(1,1){$s_0$}}
 \put(95,35){\makebox(1,1){$s_1$}}
 \put(125,35){\makebox(1,1){$s_0$}}
 \put(155,35){\makebox(1,1){$s_1$}}
 \put(185,35){\makebox(1,1){$s_0$}}
 \put(215,35){\makebox(1,1){$s_1$}}
 \put(245,35){\makebox(1,1){$s_0$}}
 \put(275,35){\makebox(1,1){$s_1$}}

 \put(18,25){\makebox(1,1){$-2Xs_1$}}
 \put(50,25){\makebox(1,1){$-2X$}}
 \put(80,25){\makebox(1,1){$s_1$}}
 \put(110,25){\makebox(1,1){$1$}}
 \put(140,25){\makebox(1,1){$2Xs_1$}}
 \put(170,25){\makebox(1,1){$2X$}}
 \put(200,25){\makebox(1,1){$4Xs_1$}}
 \put(230,25){\makebox(1,1){$4X$}}
 \put(260,25){\makebox(1,1){$6Xs_1$}}
\end{picture}

For any alcove $a$ denote the even wall of $a$ by $2 \rm wt\it(a)$; if it is the left wall then put ${\rm d}(a)=0$, 
if it is the right wall, then ${\rm d}(a)=1$.
In terms of $W$, for $w=2nX s_1^b$ one has $\rm wt\it(a)=n$, $\rm d\it(a)=b$.

An alcove walk is a sequence of simple reflections with some addition information. This sequence is called the type of walk.
Take an integer number $n$ and
consider an element  $2nX\in W$. For $n \geq 0$ the Ram-Yip formula says that $E^{A_2^{(2)}}_{-n}(x;q,t)$
($E^{{A_2^{(2)\dagger}}}_{-n}(x;q,t)$) can be constructed
from alcove walks of type $(s_1, s_0, \dots, s_1, s_0)$, $2n$ elements.
Analogously,  $E^{A_2^{(2)}}_{n}(x;q,t)$
($E^{{A_2^{(2)\dagger}}}_{n}(x;q,t)$) can be constructed
from alcove paths of type $(s_0, \dots, s_1, s_0)$, $2n-1$ elements.
Let us give definition of alcove walks in our case.

Let $w$ be an element of the affine Weyl group.
Consider a sequence $(s_{k_1}, \dots, s_{k_l})$
 of simple reflections,  such that $w=s_{k_1} \dots s_{k_l}$ is a reduced decomposition,
and a binary word $b_1, \dots, b_l$, $b_i \in \{ 0, 1 \}$.
Alcove walk of type $w$ is the path on the set of alcoves starting at the alcove $\left[ 0,1 \right]$ 
(corresponding to $1 \in W$) and consisting of the following steps:
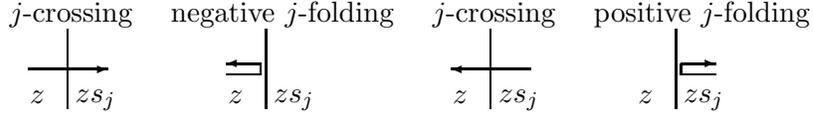
\[
\begin{picture}(300.00,50.00)(-5,00)

 \put(3,33){$j$-crossing}
 \put(25.00,00.00){\line(0,1){30}}
 \put(10.00,15.00){\vector(1,0){30}}
 \put(13,4){\makebox(1,1){$z$}}
 \put(35,3){\makebox(1,1){$z s_j$}}

 \put(64,33){negative $j$-folding}
 \put(100.00,00.00){\line(0,1){30}}
 \put(85.00,13.00){\line(1,0){13}}
 \put(98.00,17.00){\vector(-1,0){13}}
 \put(98.00,13.00){\line(0,1){4}}
 \put(88,4){\makebox(1,1){$z$}}
 \put(110,3){\makebox(1,1){$z s_j$}}
 
  \put(163,33){$j$-crossing}
 \put(185.00,00.00){\line(0,1){30}}
 \put(200.00,15.00){\vector(-1,0){30}}
 \put(173,4){\makebox(1,1){$z$}}
 \put(195,3){\makebox(1,1){$z s_j$}}

 \put(224,33){positive $j$-folding}
 \put(255.00,00.00){\line(0,1){30}}
 \put(270.00,13.00){\line(-1,0){13}}
 \put(257.00,17.00){\vector(1,0){13}}
 \put(257.00,13.00){\line(0,1){4}}
 \put(243,4){\makebox(1,1){$z$}}
 \put(265,3){\makebox(1,1){$z s_j$}}
\end{picture}\]
We have folding on the $i$-th step if $b_i=0$ and crossing if $b_i=1$.
We call a folding positive if it is "from left to right" and negative if it is "from right to left".
Put $J=\{i|b_i=0\}$.
Denote by $p_J$ the final alcove of the walk with such set of foldings. Of course, 
$p_J$ also depends on $w$ but we omit $w$ to simplify the notation.

Denote by $\mathcal{B}(w)$ the set of alcove walks of type $w$, i. e. pairs $p=(w, b)$.
Let $J_0=\{i \in J| w_i =s_0\}$ and let $J_+$ ($J_-$) be the subset of positive (correspondingly, negative) foldings of $J \backslash J_0$.

Put $\beta_i=s_{k_l} \dots s_{k_{i+1}}\alpha_{k_i}$, $1\le i\le l$, where $\alpha_{k_i}$ is a simple root.
For the walks of types $(s_1, s_0, \dots s_1, s_0)$ and $(s_1, s_0, \dots s_1, s_0)$ one has:
\begin{gather*}
\beta_{l-2i}=(s_0 s_{1})^i\alpha_0=-\alpha_1+(2i+1)\delta,\\
\beta_{l-2i-1}=(s_0 s_{1})^is_0\alpha_1=-\alpha_1+(2i+2)\delta.
\end{gather*}

Any $\beta$ from the affine root lattice of type $A_2^{(2)}$ (aka $A_1^{(1)}$) can be written in the following form:
\[\beta=\beta'+\rm deg\it(\beta)\delta,\]
where $\beta'$ is a root of finite dimensional root system. So in our case we have that ${\rm deg}(\beta_{l-i})=i+1$.

\subsection{Explicit formula in type $A_2^{(2)}$}\label{A22}
\begin{thm} \label{RYnondual}
(Ram, Yip, $A_2^{(2)}$-case) Put $v=t^{1/2}$. Let $w=(s_1, s_0, \dots, s_1, s_0)$ (with $-2n$ elements for $n<0$)
and $w=(s_0, \dots, s_1, s_0)$ (with $2n-1$ elements for $n \geq 0$). Then:
\begin{multline}
E_{n}^{A_2^{(2)} }(x,q,t)=\\
\sum_{p \in \mathcal{B}(w)}v^{(sign(n)-1)/2+d(p_J)-|J|}\left(1-v^2\right)^{|J|}
\prod_{j \in J_0}\frac{\xi_j}{1-\xi_j^2} \prod_{j \in J_+}\frac{1}{1-\xi_j} \prod_{j \in J_-}\frac{\xi_j}{1-\xi_j} x^{wt(p_J)},
\end{multline}
where $\xi_j=q^{\rm deg \it(\beta_j)}v^{-\langle \alpha_1^\vee, \beta_j \rangle}=q^{\rm deg \it(\beta_j)}v^{2}$.
\end{thm}

\begin{definition}\label{cij}
Define the elements $c_r(k_{22},k_{12},k_{11})\in \mathbb{Z}\left[q\right]$, $r=1,2$ by the following recurrent relations:
\begin{gather}\label{recurrentc1}
c_1(k_{22},k_{12},k_{11})=q^{2n}c_2(k_{22}-1,k_{12},k_{11})+q^{2n-1}c_2(k_{22},k_{12}-1,k_{11})+c_1(k_{22},k_{12},k_{11}-1),\\
\label{recurrentc2}
c_2(k_{22},k_{12},k_{11})=c_2(k_{22}-1,k_{12},k_{11})+q^{2n-1}c_2(k_{22},k_{12}-1,k_{11})+c_1(k_{22},k_{12},k_{11}-1),
\end{gather}
where $n=k_{11}+k_{12}+k_{22}$ in both formulas.
The initial values are fixed by $c_r(k_{22},k_{12},k_{11})=0$, if some $k_{ij}$ is negative, and $c_r(0,0,0)=1$.
\end{definition}

\begin{prop}\label{nondualpolynomials}
The specializations of nonsymmetric Macdonald polynomials of the type $A_2^{(2)}$ can be written in the following way ($n \geq 0$):
\begin{equation}\label{nondualnegativezero}
E_{-n}^{A_2^{(2)} }(x,q,0)=\sum_{k_{22}+k_{12}+k_{11}=n}c_2(k_{22},k_{12},k_{11})x^{k_{22}-k_{11}}.
\end{equation}
\begin{equation}\label{nondualpositivezero}
E_{n+1}^{A_2^{(2)} }(x,q,0)=\sum_{k_{22}+k_{12}+k_{11}=n+1}\left(q^{2n+1}c_2(k_{22}-1,k_{12},k_{11})+c_1(k_{22},k_{12}-1,k_{11})\right)x^{k_{11}-k_{22}+1}.
\end{equation}
\begin{equation}\label{nondualnegativeinfty}
E_{-n}^{A_2^{(2)} }(x,q^{-1},\infty)=\sum_{k_{22}+k_{12}+k_{11}=n}c_1(k_{22},k_{12},k_{11})x^{k_{11}-k_{22}}.
\end{equation}
\begin{equation}\label{nondualpositiveinfty}
E_{n+1}^{A_2^{(2)} }(x,q^{-1},\infty)=\sum_{k_{22}+k_{12}+k_{11}=n}c_2(k_{22},k_{12},k_{11})x^{k_{11}-k_{22}+1}.
\end{equation}
In particular, $E_{n+1}^{A_2^{(2)}}(x,q^{-1},\infty)=xE_{-n}^{A_2^{(2)} }(x,q,0)$.
\end{prop}

\begin{proof}
We first prove \eqref{nondualnegativezero}, \eqref{nondualpositivezero} using Theorem \ref{RYnondual}.
Let $l$ be a length of an element $w$.
We know that $\beta_{l+1-j}=-\alpha_1 +j \delta$. Therefore
$\xi_{l+1-j}=q^j v^2=q^j t$. Hence if we study specialization at $t=0$ we can put all denominators to be equal to $1$:
\[E_{n}^{A_2^{(2)} }(x,q,0)=\lim_{v \rightarrow 0}
\sum_{p \in \mathcal{B}(w)}v^{(sign(n)-1)/2+d(p_J)-|J|}
\prod_{j \in J_0}\xi_j \prod_{j \in J_-}\xi_jx^{wt(p_J)}=\]
\[=\lim_{v \rightarrow 0}
\sum_{p \in \mathcal{B}(w)}v^{(sign(n)-1)/2+d(p_J)-|J|+ 2|J_0 \cup J_-|}q^{\sum_{i \in J_0 \cup J_-}i}x^{wt(p_J)}.\]
We claim that the exponent 
${(sign(n)-1)/2+d(p_J)-|J|+ 2|J_0 \cup J_-|}$ vanishes iff there are no positive $0$-foldings
and it is positive if such foldings exist. In fact, let $J_{0+}$ and  $J_{0-}$ be the sets of positive and negative zero foldings. Then:
\begin{gather}\label{J0+}
{(sign(n)-1)/2+d(p_J)-|J|+ 2|J_0 \cup J_-|}-2|J_{0+}|=\\
\nonumber {(sign(n)-1)/2+d(p_J)-|J_+|-|J_{0+}|+ |J_{0-} \cup J_-|}=0.
\end{gather}
Therefore, ${(sign(n)-1)/2+d(p_J)-|J|+ 2|J_0 \cup J_-|}=2|J_{0+}|$.
Denote the set of paths with $|J_{0+}|=0$ by $\mathcal{QB}$.
Thus we have:
\[
E_{n}^{A_2^{(2)}}(x,q,0)=
\sum_{p \in \mathcal{QB}(w)}q^{\sum_{i \in J_0 \cup J_-}i}x^{wt(p_J)}.\]
Now let us write an alcove path in the following form. We encode it by a sequence ${\bf h}=(h_0, \dots, h_l)$
(see \cite{HHL}) of $1$'s and $2$'s. We put $h_0=1$ if
$n>0$ and $h_0=2$ when $n<0$. If $i$-th step of the path is finished by right arrow (i.e. it is a crossing from left to right or a positive folding)
then $h_i=1$. If the $i$-th step is finished by the left arrow then $h_i=2$. Then subsequences $12$ correspond to negative foldings and subsequences $21$ 
correspond to positive.
We consider the sequence $(h_0, \dots, h_l)$ as a sequence of pairs $11$, $12$, $21$, $22$ and possibly the left most element without pair.
Then the set of sequences of pairs with no pair $21$ inside corresponds to $\mathcal{QB}(w)$. Denote the set of such sequences by $\mathcal{QS}(n)$.

For any sequence ${\bf h}$ of length $l+1$ denote ${\rm leg}({\bf h})=\sum_{h_{l-j}=1, h_{l-j+1}=2} j$.
Then:
\[
E_{n}^{A_2^{(2)}}(x,q,0)=
\sum_{{\bf h} \in \mathcal{QS}(n)}q^{{\rm leg}({\bf h})}x^{\left[\left(|\{i>0|h_i=1\}|-|\{i>0|h_i=2\}|+1\right)/2\right]},\]
where $\left[y \right]$ is the integer part of $y$.

Let us consider the polynomial $E_{-n}^{A_2^{(2)}}(x,q,0)$.
Denote by $\mathcal{QS}(i,k_{22},k_{12},k_{11})\subset \mathcal{QS}(-n)$ the set of sequences of pairs and the element $h_0$ such that
$h_0=i$ and there are $k_{ij}$ pairs $ij$. Put
\[c_i(k_{22},k_{12},k_{11})=\sum_{{\bf h} \in \mathcal{QS}(i,k_{22},k_{21},k_{11})}q^{{\rm leg}({\bf h})}.\]
Let us subdivide the sets $\mathcal{QS}(i,k_{22},k_{12},k_{11})$ into three subsets of sequences according to the value of the first pair $h_1, h_2$
($(2,2)$, $(1,2)$ or $(1,1)$).
Consider the case $i=2$. If $(h_1, h_2)=(2, 2)$, then the leg
of this sequence will not be changed if we cut $(h_0, h_1)$. It is easy to see that all elements of 
$\mathcal{QS}(2,k_{22}-1,k_{12},k_{11})$ can be obtained by such a procedure. If $(h_1, h_2)=(1,2)$ then if we cut first two elements
then the leg decreases on $2l-1$. If $(h_1, h_2)=(1,1)$ then if we cut first two elements
then the leg will not be changed and we obtain $i=1$ instead of $i=2$. So we obtain a recurrent relation \eqref{recurrentc2}.
Recurrent relations \eqref{recurrentc1} can be obtained in the same way. Moreover by definition we have that they satisfy
\eqref{nondualnegativezero}.

Analogously we obtain equation \eqref{nondualpositivezero}.

Now let us consider the polynomials $E_{n}^{A_2^{(2)}}(x,q^{-1}, \infty)$.
At $t \rightarrow \infty$ $\frac{\xi_i}{1-\xi_i^2} \thicksim \frac{1}{1-\xi_i}\thicksim -\xi_i^{-1}$.
After interchanging $q \rightarrow q^{-1}$ we have:
\[E_{-n}^{A_2^{(2)} }(x,q^{-1},\infty)=\lim_{t \rightarrow \infty}
\sum_{p \in \mathcal{B}(w)}v^{(sign(-n)-1)/2+d(p_J)+|J|}
\prod_{j \in J_0\cup J_+}(-q^jv^{-2})x^{{\rm wt}(p_J)}=\]
\[=\sum_{p \in \mathcal{B}(w)}v^{(-sign(n)-1)/2+d(p_J)+|J|-2|J_0|-2|J_+|}q^{\sum_{j \in J_0\cup J_+}j}
x^{{\rm wt}(p_J)}.\]

Similar to \eqref{J0+} we have that $(sign(n)-1)/2+d(p_J)+|J|-2|J_0|-2|J_+|=0$ iff $J_{0-}=\emptyset$. So we obtain:
\[E_{-n}^{A_2^{(2)} }(x,q^{-1},\infty)=\sum_{p \in \mathcal{B}(w), |J_{0-}|=0}q^{\sum_{j \in J_0\cup J_+}j}
x^{\left[\left(|\{i>0|s_1=1\}|-|\{i>0|s_1=2\}|+1\right)/2\right]}.\]
So in terms of sequences we have that a walk giving nonzero summand correspond to a sequence of pairs
$22$, $21$, $11$ with $h_0=2$. Denote the set of such sequences by $\mathcal{QS}'(n)$ and
put ${\rm leg'}({\bf h})=\sum_{h_{l-j}=1,h_{l-j-1}=2}j$.
Then:
\[E_{-n}^{A_2^{(2)} }(x,q^{-1},\infty)=\sum_{{\bf h} \in \mathcal{QS}'(n)}q^{{\rm leg}'({\bf h})}
x^{\left[\left(|\{i>0|h_i=1\}|-|\{i>0|h_i=2\}|+1\right)/2\right]}.
\]
Now to obtain \eqref{nondualnegativeinfty} we interchange $1$ and $2$ in all definitions of the previous paragraph.

Finally:
\[E_{n}^{A_2^{(2)} }(x,q^{-1},\infty)=\lim_{t \rightarrow \infty}
\sum_{p \in \mathcal{B}(w)}v^{(d(p_J)+|J|}
\prod_{j \in J_0\cup J_+}(-q^jv^{-2})x^{{\rm wt}(p_J)}=\]
\[=\sum_{p \in \mathcal{B}(w)}v^{d(p_J)+|J|-2|J_0|-2|J_+|}q^{\sum_{j \in J_0\cup J_+}j}x^{{\rm wt}(p_J)},\]
and we obtain \eqref{nondualpositiveinfty}.
\end{proof}

\begin{lem}\label{solutionofrecurrent}
The unique solution for the quantities $c_r(k_{22},k_{12},k_{11})$ from Definition \ref{cij} is given by the formulas: 
\begin{equation}\label{cpositivezero}
c_1(k_{22},k_{12},k_{11})=q^{k_{12}^2+2k_{22}}\left(
  \begin{array}{c}
    k_{22}+k_{12}+k_{11} \\
    k_{22},k_{12},k_{11} \\
  \end{array}
\right)_{q^2},
\end{equation}
\begin{equation}\label{cnegativezero}
c_2(k_{22},k_{12},k_{11})=q^{k_{12}^2}\left(
  \begin{array}{c}
    k_{22}+k_{12}+k_{11} \\
    k_{22},k_{12},k_{11} \\
  \end{array}
\right)_{q^2}
\end{equation}
\end{lem}
\begin{proof}
Direct computation.
\end{proof}

\subsection{Dual Macdonald polynomials}
In this section we work with Macdonald polynomials of type ${A_2^{(2)\dagger}}$. We keep the notation from 
subsection \ref{RY} and \ref{A22}.

\begin{thm}
(Ram, Yip, ${A_2^{(2)\dagger}}$-case) Put $v=t^{1/2}$. Then:
\begin{multline*}
E_{n}^{{A_2^{(2)\dagger}}}(x,q,t)=
\sum_{p \in \mathcal{B}(w)}v^{(sign(n)-1)/2+d(p_J)-|J|}\left(1-v^2\right)^{|J|}\times\\
\prod_{j \in J_{0+}}\frac{1}{1-\xi_j^2}\prod_{j \in J_{0-}}\frac{\xi_j^2}{1-\xi_j^2} \prod_{j \in J_+}\frac{1}{1-\xi_j} \prod_{j \in J_-}\frac{\xi_j}{1-\xi_j} x^{wt(p_J)}.
\end{multline*}
\end{thm}

\begin{definition}\label{dualcij}
We define elements $c_r^{\dagger}(k_{22},k_{21},k_{11})\in \mathbb{Z}[q], r=1,2$ by the following recurrent relations:
\begin{gather}\label{dualrecurrentc1}
c_1^{\dagger}(k_{22},k_{21},k_{11})=q^{2n}c_2^{\dagger}(k_{22}-1,k_{21},k_{11})+q^{2n}c_1^{\dagger}(k_{22},k_{21}-1,k_{11})+c_1^{\dagger}(k_{22},k_{21},k_{11}-1),\\
\label{dualrecurrentc2}
c_2^{\dagger}(k_{22},k_{21},k_{11})=c_2^{\dagger}(k_{22}-1,k_{21},k_{11})+c_1^{\dagger}(k_{22},k_{21}-1,k_{11})+c_1^{\dagger}(k_{22},k_{21},k_{11}-1),
\end{gather}
where $n=k_{11}+k_{21}+k_{22}$.
The initial conditions are $c_r^{\dagger}(k_{22},k_{21},k_{11})=0$ if any $k_{ij}<0$ and $c_r^{\dagger}(0,0,0)=1$.
\end{definition}

\begin{prop}\label{dualpolynomials}
We have the following equations ($n \geq 0$):
\begin{equation}\label{dualnegativezero}
E_{-n}^{{A_2^{(2)\dagger}}}(x,q,0)=\sum_{k_{22}+k_{21}+k_{11}=n}c_2^{\dagger}(k_{22},k_{21},k_{11})x^{k_{11}-k_{22}}.
\end{equation}
\begin{equation}\label{dualpositivezero}
E_{n+1}^{{A_2^{(2)\dagger}}}(x,q,0)=\sum_{k_{22}+k_{21}+k_{11}=n}c_1^{\dagger}(k_{22},k_{21},k_{11})x^{k_{11}-k_{22}+1}.
\end{equation}
\begin{equation}\label{dualnegativeinfty}
E_{-n}^{{A_2^{(2)\dagger}}}(x,q^{-1},\infty)=\sum_{k_{22}+k_{21}+k_{11}=n}c_1^{\dagger}(k_{22},k_{21},k_{11})x^{k_{11}-k_{22}}.
\end{equation}
\begin{equation}\label{dualpositiveinfty}
E_{n+1}^{{A_2^{(2)\dagger}}}(x,q^{-1},\infty)=
\sum_{k_{22}+k_{21}+k_{11}=n}\left(c_2^{\dagger}(k_{22},k_{21},k_{11})+c_1^{\dagger}(k_{22},k_{21},k_{11})\right)x^{k_{11}-k_{22}+1}.
\end{equation}
In particular $E_{n+1}^{{A_2^{(2)\dagger}}}(x,q,0)=x E_{-n}^{{A_2^{(2)\dagger}}}(x,q^{-1},\infty)$.
\end{prop}
\begin{proof}
The proof is completely analogous to the proof of Proposition~\ref{nondualpolynomials}.
We have the same elements $\xi_j=q^j v^2=q^j t$. Hence if we study the specialization at $t=0$ we can put all denominators to be equal to $1$:
\[E_{n}^{{A_2^{(2)\dagger}}}(x,q,0)=\lim_{v \rightarrow 0}
\sum_{p \in \mathcal{B}(w)}x^{wt(p_J)}v^{(sign(n)-1)/2+d(p_J)-|J|}
\prod_{j \in J_0}\xi_j \prod_{j \in J_-}\xi_j=\]
\[=\lim_{v \rightarrow 0}
\sum_{p \in \mathcal{B}(w)}x^{wt(p_J)}v^{(sign(n)-1)/2+d(p_J)-|J|+ 2|J_+|+ 4 |J_{0+}|}q^{\sum_{i \in J_0 \cup J_-}i}.\]
The exponent ${(sign(n)-1)/2+d(p_J)-|J|+ 2|J_+|+ 4 |J_{0+}|}$ vanishes iff there are no negative $0$-foldings. 
Encode alcove paths by binary words in the same way as in the proof of Proposition \ref{nondualpolynomials}.
So the nonzero summands correspond to words with no pairs $12$. This is the only difference with the previous proof.
\end{proof}

\begin{lem}\label{dualsolutionofrecurrent}
The unique solution for the quantities $c_r^{\dagger}(k_{22},k_{12},k_{11})$ from Definition \ref{dualcij} is given by the formulas:
\begin{equation}\label{cdualpositivezero}
c_1^{\dagger}(k_{22},k_{21},k_{11})=q^{k_{21}(k_{21}-1)+2 k_{22}+2k_{21}}\left(
  \begin{array}{c}
    k_{22}+k_{21}+k_{11} \\
    k_{22},k_{21},k_{11} \\
  \end{array}
\right)_{q^2},
\end{equation}
\begin{equation}\label{cdualnegativezero}
c_2^{\dagger}(k_{22},k_{21},k_{11})=q^{k_{21}(k_{21}-1)}\left(
  \begin{array}{c}
    k_{22}+k_{21}+k_{11} \\
    k_{22},k_{21},k_{11} \\
  \end{array}
\right)_{q^2}
\end{equation}
\end{lem}
\begin{proof}
Direct computation.
\end{proof}

\section{Comparison}\label{WM}
In this section we establish a link between the characters of the Weyl modules and the specialized Macdonald polynomials.
\begin{thm}
For any $n\in\bZ$ one has  
\[
{\rm ch} W_n(x,q^2)=E^{{A_2^{(2)\dagger}}}_n(x,q,0),\ {\rm ch} W^\sigma_n(x,q)=E^{A_2^{(2)}}_n(x,q,0).
\]
\end{thm}
\begin{proof}
Lemma \ref{-0} and Lemma \ref{tw-0} give for $n\ge 0$
\begin{gather}\label{comp-0}
{\rm ch} W_{-n}(x,q^2)=\sum_{k+s\le n} q^{k(k-1)} x^{-n+k+2s} \left(\begin{array}{c} n\\ k,s,n-k-s\\ \end{array}\right)_{q^2},\\
\label{comptw-0}
{\rm ch} W^\sigma_{-n}(x,q)=\sum_{k+s\le n} q^{k^2}x^{-n+k+2s} 
\left(\begin{array}{c} n\\ k,s,n-k-s\\ \end{array}\right)_{q^2}.
\end{gather}
Formula \eqref{comp-0} agrees with formulas \eqref{dualnegativezero},  \eqref{cdualnegativezero}.
Formula \eqref{comptw-0} agrees with formulas \eqref{nondualnegativezero}, \eqref{cnegativezero}.

Using Corollary \ref{+0} and Corollary \ref{tw+0} we obtain the following formulas, $n\ge 0$
\begin{gather}
\label{comp+0}
{\rm ch} W_{n+1}(x,q^2)= \sum_{k+s\le n} q^{k(k+1)+2s} x^{n+1-k-2s} \left(\begin{array}{c} n\\ k,s,n-k-s\\ \end{array}\right)_{q^2},\\
\label{comptw+0}
{\rm ch} W^\sigma_{n+1}(x,q)=
\sum_{k+s\le n} x^{n+1-k-2s}\left(q^{k^2+2s}+x^{-1}q^{k^2+2n+1}\right) \left(\begin{array}{c} n\\ k,s,n-k-s\\ \end{array}\right)_{q^2}.
\end{gather}
Formula \eqref{comp+0} agrees with formulas \eqref{dualpositivezero},  \eqref{cdualpositivezero}.

The nontrivial part is formula \eqref{comptw+0}, which we want to compare with formulas 
\eqref{nondualpositivezero}, \eqref{cpositivezero} and \eqref{cnegativezero}.
We have conditions \eqref{twshift1} and \eqref{twshift2} on elements of basis of $W_{n+1}$. Elements that satisfy
condition \eqref{twshift1} are parametrized by the data very similar to the parametrization data of \eqref{twbasis}. 
More precisely we obtain \eqref{twshift1} if
we increase the $t$-degree of the elements $e_a$ in \eqref{twbasis} by $2$. Thus, under the identification 
$s=k_{11}$, $k=k_{12}$, $k_{22}+k_{12}+k_{11}=n+1$, the character of the elements \eqref{twshift1} is equal to
\begin{multline*}
\sum_{k_{22}+k_{12}+k_{11}=n+1}c_2(k_{22}-1,k_{12},k_{11})q^{2k_{22}}x^{k_{11}-k_{22}+1}=\\
\sum_{k_{22}+k_{12}+k_{11}=n+1}c_1(k_{22}-1,k_{12},k_{11})x^{k_{11}-k_{22}+1}.
\end{multline*}

Analogously the elements satisfying condition \eqref{twshift2} are elements that satisfy conditions \eqref{twbasis} multiplied by
$g^-_{2n+1}$. So their character is equal to 
\[
\sum_{k_{22}+k_{12}+k_{11}=n+1}c_2(k_{22},k_{12}-1,k_{11})q^{2n+1}x^{k_{11}-k_{22}}.
\]
\end{proof}

Let us introduce the PBW filtration on the Weyl module $W_{-n}$. Namely, we define
$F_0=\bC w_{-n}$, $F_{s+1}= F_s+\fn^-[t].F_s$. The associated graded space is a cyclic module for the
algebra $\bC[e_0,e_1,\dots]\T \bigwedge(g^+_0,g^+_1,\dots)$. Let us attach degree $1$ to the variables $e_i$
and $g^+_i$. We thus obtain an additional grading on the module ${\rm gr} W_{-n}$.
We denote the character by ${\rm ch}({\rm gr}W_{-n})(x,q,t)$.   

For the twisted Weyl modules we make a similar procedure. First we pass to the graded $\bC[e_0,e_1,\dots]\T \bigwedge(g^+_0,g^+_1,\dots)$
module ${\rm gr}W^\sigma_{-n}$. Then we attach degree $1$ to the variables $e_i$ and degree $0$ to the variables $g^+_i$.
Using this new grading we define the new character depending on $x,q,t$ and denote it by
${\rm ch} ({\rm gr}W^\sigma_{-n})(x,q,t)$.

\begin{thm}
Let $n\ge 0$. Then   
\[
{\rm ch}{\rm gr}W_{-n}(x,q^2,q^2)=E^{{A_2^{(2)\dagger}}}_{-n}(x,q^{-1},\infty),\ {\rm ch}({\rm gr} W^\sigma_{-n})(x,q,q)=E^{A_2^{(2)}}_{-n}(x,q^{-1},\infty).
\]
\end{thm}
\begin{proof}
It is easy to see that the set of vectors \eqref{basis} forms basis of ${\rm gr}W_{-n}$ and the set of vectors
\eqref{twbasis} forms basis of ${\rm gr}W^\sigma_{-n}$. Hence, we derive the following formulas for the graded characters:
\begin{gather}\label{comp-infty}
{\rm ch} ({\rm gr} W_{-n})(x,q^2,q^2)=\sum_{k+s\le n} q^{k(k-1)+2k+2s} x^{-n+k+2s} \left(\begin{array}{c} n\\ k,s,n-k-s\\ \end{array}\right)_{q^2},\\
\label{comptw-infty}
{\rm ch}({\rm gr} W^\sigma_{-n})(x,q,q)=\sum_{k+s\le n} q^{k^2+s}x^{-n+k+2s} 
\left(\begin{array}{c} n\\ k,s,n-k-s\\ \end{array}\right)_{q^2}.
\end{gather}
Now formula \eqref{comp-infty} agrees with formulas \eqref{dualnegativeinfty}, \eqref{cdualpositivezero}.
Formula \eqref{comptw-infty} agrees with formulas \eqref{nondualnegativeinfty}, \eqref{cpositivezero}.
\end{proof}

\appendix
\section{Quantum Bruhat graph}\label{OrrSh} Here we briefly describe the methods of a paper \cite{OS}.
Although we don't use their techniques, the ideas of \cite{OS} are very important for the content of our paper. 

Consider the Weyl group $W=\langle s_0\rangle \star \langle s_1 \rangle$ of a root system $A_2^{(2)}$.

\begin{definition}
Let $W(Y)$ be the Coxeter group of the root system $Y$, $s_\alpha$ be a reflection in the root $\alpha$, $l$
be the length function on $W(Y)$. Then the quantum Bruhat graph is
the following ordered labelled graph:

$\cdot$ the set of vertices is $W(Y)$;

$\cdot$ we have a Bruhat arrow from $g$ to $g s_{\alpha}$, if $l(g s_{\alpha})=l(g)+1$;

$\cdot$ we have a quantum arrow from $g$ to $g s_{\alpha}$, if $l(g s_{\alpha})=l(g)-\langle 2\rho, \alpha \rangle+1$.
\end{definition}
Consider the quantum Bruhat graph of type $\widehat{A_1}$. We want to make a difference between
$s_{\alpha_1}=s_1$ and $s_{\alpha_0}=s_0$. So we have the following labeled graph on two vertices:
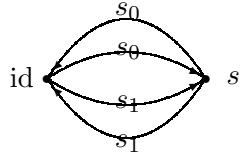
\[
\begin{picture}(140.00,50.00)(-5,00)
\put(20.00,20.00){\circle*{3.00}}
\put(80.00,20.00){\circle*{3.00}}

\put(20.00,20.00){\bezier{200}(0,0)(30,-20)(60,00)}
\put(20.00,20.00){\bezier{200}(0,0)(30,20)(60,00)}
\put(20.00,20.00){\bezier{200}(0,0)(30,-45)(60,00)}
\put(20.00,20.00){\bezier{200}(0,0)(30,45)(60,00)}

 \put(77.00,18.00){\vector(2,1){1}}
 \put(77.00,22.00){\vector(2,-1){1}}
 \put(23.00,16.00){\vector(-1,1){1}}
 \put(23.00,24.00){\vector(-1,-1){1}}
 
 \put(10,20){\makebox(1,1){$\rm id$}}
 \put(90,20){\makebox(1,1){$s$}}
 \put(50,30){\makebox(1,1){$s_0$}}
 \put(50,10){\makebox(1,1){$s_1$}}
 \put(50,-5){\makebox(1,1){$s_1$}}
 \put(50,45){\makebox(1,1){$s_0$}}
\end{picture}
\]
where arrows from $\rm id$ to $s$ are Bruhat and arrows from $s$ to $\rm id$ are quantum.

Put $\beta_i=s_{k_n} \dots s_{k_{i+1}}\alpha_{k_i}$, where $\alpha_{k_i}$ is a simple root.
We write $\beta$ in the following form $\beta=\beta'+deg(\beta)\delta$,
where $\beta' \in \mathbb{Z}\alpha$.

For any alcove walk $(w,b)$ let $J=\{i|b_i=0\}$, i. e. the set of foldings of a walk.
Then we consider the following path on the quantum Bruhat graph started at element $\rm id$:
\[
\begin{picture}(140.00,12.00)(-5,00)
 \put(10.00,2.00){\vector(1,0){30}}
 \put(70.00,2.00){\vector(1,0){30}}

 \put(20,12){\makebox(1,1){$dir(\beta_{i_1})$}}

 \put(50,2){\makebox(1,1){$\dots$}}
 \put(80,12){\makebox(1,1){$dir(\beta_{i_r})$}}
\end{picture}\]
It is easy to see that any odd arrow of this path is quantum and any even is Bruhat. The Bruhat arrows correspond
to negative foldings and quantum arrows correspond to positive ones. Define quantum Bruhat paths as paths such that they have no Bruhat
(quantum for ${A_2^{(2)\dagger}}$) edges labeled by $s_0$. It is proved in \cite{OS} that $E_n^{A_2^{(2)}}(x,q,0)$ is obtained as a
sum of some summands which are in one-to-one correspondence with the quantum Bruhat paths on the
quantum Bruhat graph.

\section*{Acknowledgments}
The Nonsymmetric Macdonald polynomials package of Sage \cite{Sage} by Anne Schilling and 
Nicolas M. Thiery was very useful for us to justify our conjectures.
The work of EF  was supported within the framework of the Academic Fund Program at the National Research University Higher School of Economics (HSE) 
in 2015--2016 (grant 15-01-0024), by the Dynasty Foundation and by the Simons foundation. 
The work of IM was supported within the framework of a subsidy granted to the HSE 
by the Government of the Russian Federation for the implementation of the Global Competitiveness Program.


\begin{thebibliography}{99}


\bibitem[A]{A} G.Andrews, {\it The Theory of Partitions}, Cambridge University Press, 1998.

\bibitem [CLS]{CLS}
L.Calixto, J.Lemay, A.Savage,
{\it Weyl modules for Lie superalgebras}, http://arxiv.org/abs/1505.06949

\bibitem [CFK]{CFK}
V.Chari, G.Fourier, T.Khandai, {\it A categorical approach to Weyl modules}, Transform. Groups,
15(3), pp. 517--549 (2010).


\bibitem [Ch1] {Ch1}
{I.~Cherednik},
{\it Nonsymmetric Macdonald polynomials},
IMRN {10} (1995), 483--515.

\bibitem [Ch2] {Ch2}
I.~Cherednik,
{\it Double affine Hecke algebras},
London Mathematical Society Lecture
Note Series, {319}, Cambridge University Press, Cambridge, 2006.

\bibitem [CF] {CF}
I.~Cherednik, E.~Feigin, {\it Extremal part of the PBW-filtration and E-polynomials},
arXiv:1306.3146.

\bibitem [CO1] {CO1}
I.~Cherednik, {D.~Orr},
{\em Nonsymmetric difference Whittaker functions},
Preprint arXiv:1302.4094v3 [math.QA] (2013).

\bibitem [CO2] {CO2}
\bysame, \bysame,
{\em One-dimensional nil-DAHA and Whittaker functions},
Transformation Groups { 18}:1 (2013), 23--59;
arXiv:1104.3918.

\bibitem [CL]{CL}
{V.~Chari}, {S.~Loktev},
{\it Weyl, Demazure and fusion modules for
the current algebra of $\msl_{r+1}$}, Adv. Math. 207 (2006), 928--960.


\bibitem[CFS]{CFS}   
V. Chari, G.Fourier, P.Senesi, {\it Weyl modules for the twisted loop algebras},
J. Algebra, 319(12), pp. 5016--5038, 2008.

\bibitem[CP]{CP}  
V. Chari, A. Pressley, {\it Weyl modules for classical and quantum affine algebras},
Represent. Theory, 5, pp. 191--223 (electronic), 2001.




\bibitem [FFL1] {FFL1}
{E.~Feigin}, and {G.~Fourier}, and {P.~Littelmann},
{\it PBW-filtration and bases for irreducible
modules in type $A_n$}, Transformation Groups { 16}:1 (2011), 71--89.

\bibitem [FFL2] {FFL2}
\bysame, \bysame, \bysame,
{\em  PBW filtration and bases for symplectic Lie algebras},
IMRN { 24} (2011), 5760--5784.

\bibitem[FM] {FM}
E.Feigin, I.Makedonskyi, {\it Nonsymmetric Macdonald polynomials, Demazure modules and PBW filtration},
Journal of Combinatorial Theory, Series A (2015), pp. 60--84.

\bibitem[FoLi1]{FoLi1} G.Fourier, P.Littelmann,
{\it Tensor product structure of affine Demazure modules and limit constructions},
Nagoya Math. Journal 182 (2006), 171--198.


\bibitem[FoLi2]{FoLi2} \bysame, and \bysame,
{\it Weyl modules, Demazure modules, KR-modules, crystals, fusion products and limit constructions},
Advances in Mathematics 211 (2007), no. 2, 566--593.


\bibitem[FeLo1]{FeLo1} B.Feigin, S. Loktev, 
{\it On generalized Kostka polynomials and the quantum Verlinde rule}, 
Differential topology, infinite-dimensional Lie algebras, and applications, 61--79, Amer. Math. Soc. Transl. Ser. 2, 194, 
Amer. Math. Soc., Providence, RI, 1999.

\bibitem[FeLo2]{FeLo2} B.Feigin, S. Loktev, {\it Multi-dimensional Weyl modules and symmetric functions}, 
Comm. Math. Phys., 251(3), pp. 427--445, 2004.


\bibitem[H]{H} M.Haiman, {\it Cherednik algebras, Macdonald polynomials and combinatorics},
Proceedings of the International Congress of Mathematicians, Madrid 2006, Vol. III, 843--872. 

\bibitem [HHL] {HHL}
{M.~Haiman}, and {J.~Haglund}, and {N.~Loehr},
{\em A combinatorial formula for non-symmetric Macdonald polynomials},
Amer. J. Math. { 130}:2 (2008), 359--383.


\bibitem[I]{I} {B.~Ion},
{\em Nonsymmetric Macdonald polynomials and Demazure characters},
Duke Mathematical Journal {116}:2 (2003), 299--318.


\bibitem[M1]{M1}
I.G.Macdonald, {\it Symmetric functions and Hall polynomials}, second ed., Oxford University Press, 1995.

\bibitem[M2]{M2}
I.G.Macdonald, {\it Affine Hecke algebras and orthogonal polynomials}, S{\' e}minaire Bourbaki, Vol. 1994/95.
Ast{\' e}risque No. 237 (1996), Exp. No. 797, 4, 189--207.


\bibitem[Mus1]{Mus1}
I.M.Musson, {\it Lie superalgebras and enveloping algebras,} vol. 131, Graduate Studies in Mathematics.
American Mathematical Society, Providence, RI, 2012.

\bibitem[Mus2]{Mus2} I.M.Musson, 
{\it The enveloping algebra of the Lie superalgebra $\mosp(1,2r)$}, Represent. Theory 1 (1997), 405--423. 


\bibitem[OS]{OS}
D.Orr, M.Shimozono, {\it Specializations of nonsymmetric Macdonald-Koornwinder polynomials},
arXiv:1310.0279.


\bibitem[P]{P} G.Pinczon, {\it The enveloping algebra of the Lie superalgebra $\mosp(1,2)$},
J. Algebra 132 (1990), no. 1, 219--242.

\bibitem[RY]{RY} A.Ram, M.Yip, {\it  A combinatorial formula for Macdonald polynomials},
Adv. Math. 226 (2011), no. 1, 309--331.

\bibitem[Sage]{Sage} SageMath, {Nonsymmetric Macdonald polynomials package} by A.~Schilling and N.~M.~Thiery (2013),
http://doc.sagemath.org/html/en/reference/combinat/sage/\-combinat/\-root\_system/\-non\_symmetric\_macdonald\_polynomials.html.

\bibitem[S]{S} {Y.~Sanderson},
{\em On the Connection Between Macdonald Polynomials and
Demazure Characters},
J. of Algebraic Combinatorics, {11} (2000), 269--275.

\bibitem[SVV]{SVV}
P. Shan, M. Varagnolo, E. Vasserot, {\it On the center of quiver-Hecke algebras},
arXiv:1411.4392.
    



\end{thebibliography}
\end{document}